\documentclass[a4paper,11pt,reqno]{amsart}
\usepackage[utf8]{inputenc}
\usepackage{amsfonts}
\usepackage{amsmath}
\usepackage{amssymb}
\usepackage{amsthm}
\usepackage{verbatim}
\usepackage{color}
\usepackage{hyperref}
\usepackage{epstopdf}

\usepackage{thmtools}
\usepackage{theoremref}
\theoremstyle{plain}
\newtheorem{theorem}{Theorem}[section]
\newtheorem{lemma}[theorem]{Lemma}
\newtheorem{corollary}[theorem]{Corollary}
\newtheorem{proposition}[theorem]{Proposition}

\theoremstyle{definition}

\ProvideTextCommand{\DJ}{OT1}{\leavevmode\raisebox{-.5ex}{\makebox[0pt][l]{\hskip-.07em\accent"16\hss}}D}

\numberwithin{equation}{section}

\title[Mixed second moment of Dirichlet $L$-functions]{The Mixed Second Moment of quadratic Dirichlet $L$-functions with prime conductor}

\author{J. MacMillan}
\address{Department of Mathematics, Swansea University, Swansea, SA1 8EN, United Kingdom}
\email{j.b.macmillan@swansea.ac.uk}

\date{\today}


\begin{document}
	
	\begin{abstract}
		We establish an asymptotic formula for the mixed second moment involving the second derivative of the completed $L$-function, when averaged over monic, irreducible polynomials in the function field setting. 
	\end{abstract}
	\maketitle
	
	\section{Introduction}
	The study of moments of various families of L-functions is a classical topic in Analytic Number Theory. In the case of the Riemann-zeta function, a problem is to understand the asymptotic behaviour of 
	\[
	M_k(T)=\frac{1}{T}\int_0^T\left|\zeta\left(\frac{1}{2}+it\right)\right|^{2k}dt. 
	\]
	In this setting the second moment was computed by Hardy and Littlewood \cite{HardyLittlewood1916} and the fourth moment was computed by Ingham \cite{Ingham1927}. For the $2k^{\text{th}}$ moment, it is conjectured that as $T\rightarrow\infty$
	\[
	M_k(T)\sim a_kg_k(\log T)^{k^2}
	\]
	where $a_k$ is an arithmetic factor in terms of an Euler product and $g_k$ is an geometric factor. Using random matrix theory, Keating and Snaith \cite{KeatingSnaith2000R} conjectured a precise value for $g_k$ and Conrey et al. \cite{CFKRS2005} developed a heuristic recipe to conjecture 
	\[
	M_k(T)\sim P_k(\log T),
	\]
	where $P_k$ is an explicit polynomial of degree $k^2$. Ingham \cite{Ingham1927} also studied the mixed moments of the Riemann-zeta function and showed that for integers $\mu,\nu\geq 0$ we have 
	\[
	\frac{1}{T}\int_0^T\zeta^{(\mu)}\left(\frac{1}{2}+it\right)\zeta^{(\nu)}\left(\frac{1}{2}-it\right)dt\sim \frac{1}{\mu+\nu+1}(\log T)^{\mu+\nu+1}
	\]
	Since $\overline{\zeta^{(\mu)}\left(\frac{1}{2}+it\right)}=\zeta^{(\mu)}\left(\frac{1}{2}-it\right)$, then Ingham's result implies
	\[
	\frac{1}{T}\int_0^T\left|\zeta^{(\mu)}\left(\frac{1}{2}+it\right)\right|^2dt\sim\frac{1}{2\mu+1}(\log T)^{2\mu+1}.
	\]
	In a recent paper, Keating and Wei \cite{KeatingWei2024} used Random matrix theory to conjecture an asymptotic formulae for the joint moments of higher order derivatives of the Riemann-zeta function. For integers $n_1,n_2 \geq 0$ and $k\geq 1$ with $1\leq M\leq k$, they conjectured
	\[
	\frac{1}{T}\int_0^T\left|\zeta^{(n_1)}\left(\frac{1}{2}+it\right)\right|^{2M}\left|\zeta^{(n_2)}\left(\frac{1}{2}+it\right)\right|^{2k-2M}dt\sim a_{k,M}(n_1,n_2)c_k(\log T)^{k^2+2Mn_1+2(k-M)n_2}
	\]
	where $a_{k,M}(n_1,n_2)$ is an explicit formulae depending on $k,M,n_1,n_2$ and $c_k$ is an arithmetic factor in terms of an Euler product. In particular, they showed that for $k=M=1$, $n_1=\mu$, $n_2=0$, they showed that their asymptotic formula agrees with Ingham's result. \\
	
	For the family of Dirichlet L-functions $L(s,\chi_d)$, with $\chi_d$ a real primitive Dirichlet character modulo $d$ defined by the Jacobi symbol $\chi_d(n)=\left(\frac{d}{n}\right)$, a problem is understanding the asymptotic behaviour of
	\[
	\frac{1}{D}\sum_{0<d\leq D}L\left(\frac{1}{2},\chi_d\right)^k,
	\]
	when summing over fundamental discriminants $d$ as $D\rightarrow\infty$. In this context Jutila \cite{Jutila1981} obtained an asymptotic formula for the first and second moment. When averaging over fundamental discriminants $8d$ Soundararajan \cite{Soundararajan2000} established an asymptotic formula for the second and third moments with a power saving error term. Assuming the Generalised Riemann hypothesis Shen \cite{Shen2021} established the first term in the asymptotic formula for the fourth moment. Subsequently Shen and Stucky \cite{ShenStucky2026} removed the condition of the Generalised Riemann hypothesis and obtained the next three terms in the asymptotic formula. It is conjectured that 
	\begin{equation}\label{conjfunddisc}
		\frac{1}{D}\sum_{0<d\leq D}L\left(\frac{1}{2},\chi_d\right)^k\sim a_kg_k(\log D)^{\frac{k(k+1)}{2}},
	\end{equation}
	where $a_k$ is an arithmetic factor in the form of a Euler product and $g_k$ is a geometric factor. Using Random matrix theory, Keating and Snaith \cite{KeatingSnaith2000L} conjectured a precise value for $g_k$ and Conrey et al. \cite{CFKRS2005} developed a heuristic to conjecture the lower terms. \\
	
	A similar problem is to study the moments of quadratic Dirichlet L-functions with prime conductor at the central point $s=\frac{1}{2}$. In this context, Jutila \cite{Jutila1981} established an asymptotic formula for the first moment, and Baluyot and Pratt \cite{BaluyotPratt2022} established an asymptotic formula for the second moment under the condition of the generalised Riemann hypothesis. A similar conjecture to (\ref{conjfunddisc}) is believed for these moments due to both families having symplectic symmetry type.\\
	
	In function fields, an analogous problem is to study the asymptotic behaviour of 
	\[
	\frac{1}{|\mathcal{H}_{2g+1}|}\sum_{D\in\mathcal{H}_{2g+1}}L\left(\frac{1}{2},\chi_D\right)^k,
	\]
	where $\mathcal{H}_{2g+1}$ denotes the space of monic, square-free polynomials of degree $2g+1$ over $\mathbb{F}_q[t]$ as $g\rightarrow\infty$. In this context, Andrade and Keating \cite{AndradeKeating2012} established an asymptotic formula for the first moment and then they \cite{AndradeKeating2014} adapted the heuristic of Conrey et al. \cite{CFKRS2005} to conjecture an asymptotic formula for all integral moments. Florea \cite{Florea2017first} improved the asymptotic formula for the first moment by establishing a secondary main term of size $gq^{-\frac{4g}{3}}$ while improving the error term. Florea \cite{Florea2017secondthird,Florea2017fourth} went on to establish an asymptotic formula for the second, third and fourth moments. The study of the moments of derivatives for these families of Dirichlet L-functions is also an interesting topic. Since the first derivative of $L(\frac{1}{2},\chi_D)$ is a constant multiple of itself (see \cite[Lemma~2.1]{BaeJung2019}), then it is natural to study the moments of the second and higher derivatives. In this setting, an asymptotic formula for the second derivative is established by Andrade and Rajagopal \cite{AndradeRajagopal2016} and then improved by Bae and Jung \cite{BaeJung2019}. A general formula for the $\mu^{\text{th}}$ derivative of $L(\frac{1}{2},\chi_D)$ was established by Andrade and Jung \cite{AndradeJung2021}.  Defining the completed L-function by $\Lambda(s,\chi_D)=q^{g(s-\frac{1}{2})}L(s,\chi_D)$, then Djankovi\'c and {\DJ}oki\'c \cite{DjankovicDjokic2021} considered the mixed second moment involving the second derivative of the completed L-function.
	\begin{theorem}[Djankovi\'c and {\DJ}oki\'c]\thlabel{DD2021}
		Let $q$ be a prime with $q\equiv 1(\text{mod }4)$. Then as $g\rightarrow\infty$ we have that
		\[
		\frac{1}{|\mathcal{H}_{2g+1}|}\sum_{D\in\mathcal{H}_{2g+1}}\frac{\Lambda''\left(\frac{1}{2},\chi_D\right)\Lambda\left(\frac{1}{2},\chi_D\right)}{(\log q)^2}=P(2g+1)+O(q^{g(\epsilon-1)}),
		\]
		where $P(x)$ is a polynomial of degree 5.
	\end{theorem}
	Similarly in function fields, a problem is to understand the asymptotic behaviour of 
	\[
	\frac{1}{|\mathcal{P}_{2g+1}|}\sum_{P\in\mathcal{P}_{2g+1}}L\left(\frac{1}{2},\chi_P\right)^k,
	\]
	when summing over all monic irreducible (prime) polynomials of degree $2g+1$ in $\mathbb{F}_q[t]$. An asymptotic formula for the first and second moments were established by Andrade and Keating \cite{AndradeKeating2013}. In particular, for the second moment, they showed that
	\[
	\frac{1}{|\mathcal{P}_{2g+1}|}\sum_{P\in\mathcal{P}_{2g+1}}L\left(\frac{1}{2},\chi_P\right)^2=\frac{g^3}{3\zeta_{\mathbb{A}}(2)}+O(g^2).
	\]
	This asymptotic formula was later improved by Bui and Florea \cite{BuiFlorea2020} who showed that 
	\begin{equation}\label{BuiFloreasecondmoment}
		\frac{1}{|\mathcal{P}_{2g+1}|}\sum_{P\in\mathcal{P}_{2g+1}}L\left(\frac{1}{2},\chi_P\right)^2=\frac{g^3}{3\zeta_{\mathbb{A}}(2)}+g^2\left(\frac{3}{2}+\frac{1}{2q}\right)+O_{\epsilon}(g^{\frac{3}{2}+\epsilon}). 
	\end{equation}
	Similar to the heuristic developed by Conrey et al. \cite{CFKRS2005}, Andrade, Jung and Shamesaldeen \cite{AndradeJungShamesaldeen2021} conjectured asymptotic formulas for the integral moments of $L(\frac{1}{2},\chi_P)$. Similarly, the moments of derivatives of these families of Dirichlet L-functions is an important topic to study. An asymptotic formula for the $\mu^{\text{th}}$ derivative of $L(\frac{1}{2},\chi_P)$ was established in \cite{Jung2022}. In a recent paper, Best \cite{Best2026} considered the mixed second moment of the product of the $\mu^{\text{th}}$ and $\nu^{\text{th}}$ derivatives of $L(\frac{1}{2},\chi_P)$. In particular for $\mu=2$ and $\nu=0$, they proved the following result. 
	\begin{theorem}[Best]\thlabel{Best2026}
		Let $q$ be an odd prime, then as $g\rightarrow\infty $ we have
		\[
		\frac{1}{|\mathcal{P}_{2g+1}|}\sum_{P\in\mathcal{P}_{2g+1}}\frac{L^{''}(\frac{1}{2},\chi_P)L(\frac{1}{2},\chi_P)}{(\log q)^2}=\frac{2g^5}{5\zeta_{\mathbb{A}}(2)}+O(g^4).
		\]
	\end{theorem}
	In this paper, we use the methods of Bui and Florea \cite{BuiFlorea2020} to establish an asymptotic formula for the mixed second moment involving the second derivative of L-functions for this family. We define the completed L-function as 
    \[
    \Lambda(s,\chi_P):=q^{g(s-\frac{1}{2})}L(s,\chi_P),
    \]
    which satisfies the symmetric function equation
    \[
    \Lambda(s,\chi_P)=\Lambda(1-s,\chi_P),
    \]
    and $\Lambda'(\frac{1}{2},\chi_P)=0$. The main result for this paper is the following, which can be seen as the monic irreducible version of \thref{DD2021}.
	\begin{theorem}\thlabel{Main}
		Let $q$ be an odd prime then as $g\rightarrow\infty$ we have  
		\[
		\frac{1}{|\mathcal{P}_{2g+1}|}  \sum_{P\in\mathcal{P}_{2g+1}}\frac{\Lambda^{''}(\frac{1}{2},\chi_P)\Lambda(\frac{1}{2},\chi_P)}{(\log q)^2}=\frac{g^5}{15\zeta_{\mathbb{A}}(2)}+g^4\left(\frac{1}{2}+\frac{1}{6q}\right)+\frac{4g^3}{3}+g^2\left(\frac{3}{2}-\frac{1}{6q}\right)+O_{\epsilon}(g^{\frac{3}{2}+\epsilon}).
		\]
	\end{theorem}
	Using \thref{Main}, (\ref{BuiFloreasecondmoment}) and the fact that 
	\[
	\sum_{P\in\mathcal{P}_{2g+1}}\frac{L''\left(\frac{1}{2},\chi_P\right)L(\frac{1}{2},\chi_P)}{(\log q)^2}=\sum_{P\in\mathcal{P}_{2g+1}}\frac{\Lambda''(\frac{1}{2},\chi_P)\Lambda(\frac{1}{2},\chi_P)}{(\log q)^2}+g^2\sum_{P\in\mathcal{P}_{2g+1}}L\left(\frac{1}{2},\chi_P\right)^2,
	\]
	we can improve the asymptotic formula in \thref{Best2026}.
	\begin{corollary}
		Let $q$ be an odd prime then as $g\rightarrow\infty$ we have  
		\[
		\frac{1}{|\mathcal{P}_{2g+1}|}  \sum_{P\in\mathcal{P}_{2g+1}}\frac{L^{''}(\frac{1}{2},\chi_P)L(\frac{1}{2},\chi_P)}{(\log q)^2}=\frac{2g^5}{5\zeta_{\mathbb{A}}(2)}+g^4\left(2+\frac{2}{3q}\right)+O_{\epsilon}(g^{\frac{7}{2}+\epsilon}).
		\]
	\end{corollary}
	
	\section{Background and Preliminary Lemmas}
	In this section, we recall the necessary background and preliminary results needed. For a general reference see \cite{Rosen2002}.
	\subsection{Background}
	Let $\mathbb{F}_q$ denote a finite field with $q$ odd. We denote by $\mathbb{A}^+, \mathbb{A}^+_n$ and $\mathbb{A}^+_{\leq n}$ to be the set of all monic polynomials, the set of all monic polynomials of degree $n$ and the set of all monic polynomials of degree at most $n$ respectively, in $\mathbb{A}=\mathbb{F}_q[t]$. Furthermore, let $\mathcal{P}$ and $\mathcal{P}_n$ denote the set of all monic irreducible polynomials and the set of all monic irreducible polynomials of degree $n$ in $\mathbb{A}$. The Prime Polynomial Theorem states that
	\[
	|\mathcal{P}_n|=\frac{q^n}{n}+O\left(\frac{q^{\frac{n}{2}}}{n}\right).
	\]
	For $f\in\mathbb{A}$, we define the norm as $|f|=q^{\deg(f)}$. The zeta-function $\zeta_{\mathbb{A}}(s)$ for $\mathbb{A}$ is defined as
	\[
	\zeta_{\mathbb{A}}(s)=\sum_{f\in\mathbb{A}^+}\frac{1}{|f|^s}=\prod_{Q\in\mathcal{P}}\left(1-\frac{1}{|Q|^s}\right)^{-1},
	\]
	for $\Re(s)>1$. Since $|\mathbb{A}^+_n|=q^n$, then 
	\[
	\zeta_{\mathbb{A}}(s)=\frac{1}{1-q^{1-s}},
	\]
	which provides a meromorphic continuation of $\zeta_{\mathbb{A}}$ with a simple pole at $s=1$. For $P\in\mathcal{P}_{2g+1}$, we define the quadratic character using the Legendre symbol $\chi_P(f)=\left(\frac{f}{P}\right)$. Then the quadratic Dirichlet L-function associated with the quadratic character $\chi_P$ is 
	\[
	L(s,\chi_P)=\sum_{f\in\mathbb{A}^+}\frac{\chi_P(f)}{|f|^s}=\prod_{Q\in\mathcal{P}}\left(1-\frac{\chi_P(Q)}{|Q|^s}\right)^{-1}.
	\]
	We will often make use of the change of variables $u=q^{-s}$, then the zeta-function is 
	\[
	\mathcal{Z}(u)=\sum_{f\in\mathbb{A}^+}u^{\deg(f)}=\prod_{Q\in\mathcal{P}}\left(1-u^{\deg(Q)}\right)^{-1}=\frac{1}{1-qu},
	\]
	with a simple pole at $u=\frac{1}{q}$. Furthermore, the Dirichlet L-function associated with $\chi_P$ is
	\[
	\mathcal{L}(u,\chi_P)=\sum_{f\in\mathbb{A}^+}\chi_P(f)u^{\deg(f)}=\prod_{Q\in\mathcal{P}}\left(1-\chi_P(Q)u^{\deg(Q)}\right)^{-1},
	\]
	where $\mathcal{L}(u,\chi_P)$ is a polynomial of degree $2g$ and satisfies the functional equation
	\[
	\mathcal{L}(u,\chi_P)=(qu^2)^g\mathcal{L}\left(\frac{1}{qu},\chi_P\right). 
	\]
	By Weil \cite{Weil1948}, all the zeros of $\mathcal{L}(u,\chi_P)$ lie on the circle $|u|=q^{-\frac{1}{2}}$. 
	\subsection{Preliminary Lemmas}
	Here we state the results needed to prove \thref{Main}. Firstly, we have the following approximate functional equations.
	\begin{lemma}\thlabel{AFELambda}
		For $P\in\mathcal{P}_{2g+1}$, we have 
		\[
		\frac{\Lambda^{''}\left(\frac{1}{2},\chi_P\right)\Lambda\left(\frac{1}{2},\chi_P\right)}{(\log q)^2}=\sum_{f\in\mathbb{A}^+_{\leq 2g-1}}\frac{\tau(f)\chi_P(f)}{\sqrt{|f|}}(2g-\deg(f))^2.
		\]
	\end{lemma}
	\begin{proof}
		The proof is similar to that given in \cite[Lemma~2.2]{DjankovicDjokic2021}.
	\end{proof}
	\begin{lemma}\thlabel{AFE}
		For $P\in\mathcal{P}_{2g+1}$, we have
		\[
		\mathcal{L}\left(\frac{u}{\sqrt{q}},\chi_P\right)^2=\sum_{f\in\mathbb{A}^+_{\leq 2g}}\frac{\tau(f)\chi_P(f)u^{\deg(f)}}{\sqrt{|f|}}+u^{4g}\sum_{f\in\mathbb{A}^+_{\leq 2g-1}}\frac{\tau(f)\chi_P(f)}{\sqrt{|f|}u^{\deg(f)}},
		\]
		where $\tau(f)=\sum_{f_1f_2=f}1$ is the divisor function. 
	\end{lemma}
	\begin{proof}
		See \cite{AndradeKeating2013}, equation (4.4).    
	\end{proof}
	Next, we need the following Weil bound on character sum over monic irreducible polynomials. 
	\begin{lemma}\thlabel{Weilbound}
		For $f\in\mathbb{A}^+$ not a square, we have 
		\[
		\frac{1}{|\mathcal{P}_{2g+1}|}\sum_{P\in\mathcal{P}_{2g+1}}\chi_P(f)\ll q^{-g}\deg(f).
		\]
	\end{lemma}
	\begin{proof}
		See \cite{Rudnick2010}, equation (2.5).
	\end{proof}
	Also, we have the following upper bound for moments of Dirichlet L-functions. 
	\begin{corollary}\cite[Corollary~4.3]{BuiFlorea2020}\thlabel{Corollarybound}
		Let $u=e^{i\theta}$ with $\theta\in[0,2\pi)$. Then
		\[
		\frac{1}{|\mathcal{P}_{2g+1}|}\sum_{P\in\mathcal{P}_{2g+1}}\left|\mathcal{L}\left(\frac{u}{\sqrt{q}},\chi_P\right)\right|^2\ll_{\epsilon} g^{1+\epsilon}\min\left\{g,\frac{1}{\overline{2\theta}}\right\}^2,
		\]
		where $\overline{\theta}=\min\{\theta,2\pi-\theta\}$. 
	\end{corollary}

	\section{The tail}
	In this section we prove the following result.
	\begin{proposition}\thlabel{E(X)}
		For $X<g$, let 
		\[
		E(X)=\frac{1}{|\mathcal{P}_{2g+1}|}\sum_{P\in\mathcal{P}_{2g+1}}\sum_{2X-1<\deg(f)\leq 2g-1}\frac{\tau(f)\chi_P(f)}{\sqrt{|f}|}(2g-\deg(f))^2.
		\]
		Then 
		\[
		E(X)=\frac{2g^2(g-X)^3}{3\zeta_{\mathbb{A}}(2)}+\frac{g^2(g-X)^2}{\zeta_{\mathbb{A}}(2)}+\frac{g^2(g-X)}{3\zeta_{\mathbb{A}}(2)}+O_{\epsilon}(g^{\frac{3}{2}+\epsilon}(g-X)^3)+O(g^{\frac{1}{2}}(g-X)^5).
		\]
	\end{proposition}
	\begin{proof}
		We apply the function field analogue of Perron's formula to get
		\begin{align*}
			E(X)&=\frac{1}{2\pi i}\oint_{|u|=1}\frac{1}{|\mathcal{P}_{2g+1}|}\sum_{P\in\mathcal{P}_{2g+1}}\sum_{f\in\mathbb{A}^+}\frac{\tau(f)\chi_P(f)}{\sqrt{|f|}}\sum_{n=2X}^{2g-1}\frac{(2g-n)^2}{u^n}\frac{du}{u}\\
			&=\frac{1}{2\pi i}\oint_{|u|=1}\frac{1}{|\mathcal{P}_{2g+1}|}\sum_{P\in\mathcal{P}_{2g+1}}\mathcal{L}\left(\frac{u}{\sqrt{q}},\chi_P\right)^2\sum_{j=1}^{2g-2X}j^2u^j\frac{du}{u^{2g+1}}.
		\end{align*}
		Write the circular contour as $C_1+C_2+C_3+C_4$, where $C_1$ denotes the arc of angle $2\pi\theta_1$ centered around $u=1$, $C_2$ denotes the arc of angle $2\pi\theta_1$ around $u=-1$, where $\frac{1}{g}\ll \theta_1=o(1)$ and $C_3+C_4$ the complement of $C_1+C_2$. Let 
		\[
		E_i(X)=\frac{1}{2\pi i}\int_{C_i}\frac{1}{|\mathcal{P}_{2g+1}|}\sum_{P\in\mathcal{P}_{2g+1}}\mathcal{L}\left(\frac{u}{\sqrt{q}},\chi_P\right)^2\sum_{j=1}^{2g-2X}j^2u^j\frac{du}{u^{2g+1}},
		\]
		then 
		\begin{equation}\label{sumE}
			E(X)=E_1(X)+E_2(X)+E_3(X)+E_4(X).
		\end{equation}
		Note that 
		\[
		\left|\sum_{j=1}^{2g-2X}j^2u^j\right|\leq 8(g-X)^3
		\]
		so there is no pole at $u=1$. By \thref{Corollarybound}, we have for $i=3,4$
		\begin{equation}\label{E3E4bound}
			E_i(X)\ll_\epsilon g^{1+\epsilon}(g-X)^3\theta_1^{-1}.
		\end{equation}
		From \thref{AFE} and \thref{Weilbound}, we have 
		\[
		\frac{1}{|\mathcal{P}_{2g+1}|}\sum_{P\in\mathcal{P}_{2g+1}}\mathcal{L}\left(\frac{u}{\sqrt{q}},\chi_P\right)^2=\sum_{f\in\mathbb{A}^+_{\leq g }}\frac{\tau(f^2)u^{2\deg(f)}}{|f|}+u^{4g}\sum_{f\in\mathbb{A}^+_{\leq g-1}}\frac{\tau(f^2)}{|f|u^{2\deg(f)}}+O(g^2),
		\]
		and so for $i=1,2$ 
		\begin{align*}
			E_i(X)&=\frac{1}{2\pi i}\int_{C_i}\sum_{f\in\mathbb{A}^+_{\leq g}}\frac{\tau(f^2)}{|f|}\sum_{j=1}^{2g-2X}j^2u^j\frac{du}{u^{2g-2\deg(f)+1}}\\
			&+\frac{1}{2\pi i}\int_{C_i}\sum_{f\in\mathbb{A}^+_{\leq g-1}}\frac{\tau(f^2)}{|f|}\sum_{j=1}^{2g-2X}j^2u^j\frac{du}{u^{-2g+2\deg(f)+1}}+O(g^2(g-X)^3\theta_1).
		\end{align*}
		Applying the function field analogue of Perron's formula again gives
		\begin{align*}
			E_i(X)&=\frac{1}{(2\pi i)^2}\oint_{|v|=r}\int_{C_i}\frac{1}{(1-v)}\sum_{f\in\mathbb{A}^+}\frac{\tau(f^2)}{|f|}(vu^2)^{\deg(f)}\sum_{j=1}^{2g-2X}j^2u^j\frac{du}{u^{2g+1}}\frac{dv}{v^{g+1}}\\
			&+\frac{1}{(2\pi i)^2}\oint_{|v|=r}\int_{C_i}\frac{1}{(1-v)}\sum_{f\in\mathbb{A}^+}\frac{\tau(f^2)}{|f|}\left(\frac{v}{u^2}\right)^{\deg(f)}\sum_{j=1}^{2g-2X}j^2u^j\frac{du}{u^{-2g+1}}\frac{dv}{v^g}+O(g(g-X)^3\theta_1)
		\end{align*}
		for $r<1$. Since 
		\begin{equation}\label{sumtau}
			\sum_{f\in\mathbb{A}^+}\tau(f^2)w^{\deg(f)}=\frac{\mathcal{Z}(w)^3}{\mathcal{Z}(w^2)}=\frac{(1-qw^2)}{(1-qw)^3},
		\end{equation}
		then, using a change of variables, we have
		\begin{align*}
			E_i(X)&=\frac{1}{(2\pi i)^2}\oint_{|w|=r}\int_{C_i}\frac{(1-\frac{w^2}{q})}{(1-w)^3(1-\frac{w}{u^2})}\sum_{j=1}^{2g-2X}j^2u^j\frac{du}{u}\frac{dw}{w^{g+1}}\\
			&+\frac{1}{(2\pi i)^2}\oint_{|w|=r}\int_{C_i}\frac{(1-\frac{w^2}{q})}{(1-w)^3(1-wu^2)}\sum_{j=1}^{2g-2X}j^2u^judu\frac{dw}{w^g}+O(g(g-X)^3\theta_1).
		\end{align*}
		Using the change of variables $u=e^{2\pi i\theta}$, we have
		\begin{align*}
			\frac{1}{2\pi i}\int_{C_1}\frac{1}{u(1-\frac{w}{u^2})}\sum_{j=1}^{2g-2X}j^2u^jdu&=\sum_{j=1}^{2g-2X}j^2\sum_{n=0}^{\infty}w^n\frac{1}{2\pi i}\int_{C_1}u^{j-2n-1}du\\
			&=\sum_{j=1}^{2g-2X}j^2\sum_{n=0}^{\infty}w^n\int_{-\frac{\theta_1}{2}}^{\frac{\theta_1}{2}}e^{2\pi i\theta(j-2n)}d\theta\\
			&=\frac{1}{\pi}\sum_{j=1}^{2g-2X}j^2\sum_{n=0}^{\infty}w^n\frac{\sin((2n-j)\pi\theta_1)}{2n-j}.
		\end{align*}
		Similarly, using the change of variables $u=-e^{2\pi i\theta}$ we have 
		\begin{align*}
			\frac{1}{2\pi i}\int_{C_2}\frac{1}{u(1-\frac{w}{u^2})}\sum_{j=1}^{2g-2X}j^2u^jdu&=\sum_{j=1}^{2g-2X}j^2\sum_{n=0}^{\infty}w^n\frac{1}{2\pi i}\int_{C_2}u^{j-2n-1}du\\
			&=\sum_{j=1}^{2g-2X}j^2\sum_{n=0}^{\infty}w^n\int_{-\frac{\theta_1}{2}}^{\frac{\theta_1}{2}}(-e^{2\pi i \theta})^{(j-2n)}d\theta\\
			&=\frac{1}{\pi}\sum_{j=1}^{2g-2X}(-1)^jj^2\sum_{n=0}^{\infty}w^n\frac{\sin((2n-j)\pi\theta_1)}{2n-j}.
		\end{align*}
		Therefore 
		\begin{align*}
			&\frac{1}{2\pi i}\int_{C_1}\frac{1}{u(1-\frac{w}{u^2})}\sum_{j=1}^{2g-2X}j^2u^jdu+\frac{1}{2\pi i}\int_{C_2}\frac{1}{u(1-\frac{w}{u^2})}\sum_{j=1}^{2g-2X}j^2u^jdu\\&=\frac{1}{\pi}\sum_{j=1}^{2g-2X}(1+(-1)^j)j^2\sum_{n=0}^{\infty}w^n\frac{\sin((2n-j)\pi \theta_1)}{2n-j}.
		\end{align*}
		Using the fact that 
		\[
		1+(-1)^j=\begin{cases}
			2 &\text{if }j \text{ is even},\\ 0 & \text{if }j \text{ is odd}.
		\end{cases}
		\]
		Then 
		\begin{align}\label{C11C2}
			\frac{1}{\pi}\sum_{j=1}^{2g-2X}(1+(-1)^j)j^2\sum_{n=0}^{\infty}w^n\frac{\sin((2n-j)\pi \theta_1)}{2n-j}&=\frac{2}{\pi}\sum_{\substack{j=1\\j\text{ even}}}^{2g-2X}j^2\sum_{n=0}^{\infty}w^n\frac{\sin((2n-j)\pi\theta_1)}{2n-j}\nonumber\\
			&=\frac{4}{\pi}\sum_{j=1}^{g-X}j^2\sum_{n=0}^{\infty}w^n\frac{\sin(2\pi(n-j)\theta_1)}{n-j}.
		\end{align}
		Similarly, we have 
		\begin{align}\label{C12C2}
			&\frac{1}{2\pi i}\int_{C_1}\frac{uw}{(1-wu^2)}\sum_{j=1}^{2g-2X}j^2u^jdu+\frac{1}{2\pi i}\int_{C_2}\frac{uw}{(1-wu^2)}\sum_{j=1}^{2g-2X}j^2u^jdu\nonumber\\
			&=\frac{4}{\pi}\sum_{j=0}^{g-X}j^2\sum_{n=1}^{\infty}w^n\frac{\sin(2\pi(n+j)\theta_1)}{n+j}.
		\end{align}
		Therefore, combining (\ref{sumE}), (\ref{E3E4bound}), (\ref{C11C2}) and (\ref{C12C2}) we have 
		\begin{align*}
			E(X)&=\frac{4}{\pi}\frac{1}{2\pi i}\oint_{|w|=r}\sum_{j=1}^{g-X}j^2\left(\sum_{n=0}^gw^n\frac{\sin(2\pi(n-j)\theta_1)}{n-j}+\sum_{n=1}^gw^n\frac{\sin(2\pi(n+j)\theta_1)}{n+j}\right)\\
			&+O(g^2(g-X)^3\theta)+O_{\epsilon}(g^{1+\epsilon}(g-X)^3\theta_1^{-1}).
		\end{align*}
		Enlarging the contour and evaluating the residue at $w=1$, we have 
		\begin{align}
			E(X)&=\frac{2}{\pi}\left(\sum_{j=1}^{g-X}j^2\sum_{n=0}^g\frac{\sin(2\pi(n-j)\theta_1)}{n-j}P(n)+\sum_{j=1}^{g-X}j^2\sum_{n=1}^g\frac{\sin(2\pi(n+j)\theta_1)}{n+j}P(n)\right)\nonumber\\
			&+O(g^2(g-X)^3\theta_1)+O_{\epsilon}(g^{1+\epsilon}(g-X)^3\theta_1^{-1}),
		\end{align}
		where 
		\begin{align*}
			P(x)=\left(1-\frac{1}{q}\right)g^2-2\left(1-\frac{1}{q}\right)gx+\left(1-\frac{1}{q}\right)x^2+\left(3+\frac{1}{q}\right)g-\left(3+\frac{1}{q}\right)x+2.
		\end{align*}
		Invoking \thref{weightedlemma} with $\alpha=g-X$, we have 
		\begin{align}
			E(X)&=\frac{2g^2(g-X)^3}{3\zeta_{\mathbb{A}}(2)}+\frac{g^2(g-X)^2}{\zeta_{\mathbb{A}}(2)}+\frac{g^2(g-X)}{3\zeta_{\mathbb{A}}(2)}\nonumber\\
			&+O(g^2(g-X)^3\theta_1)+O_{\epsilon}(g^{1+\epsilon}(g-X)^3\theta_1^{-1})+O(g(g-X)^3\theta_1^{-1})+O(g(g-X)^5\theta_1).
		\end{align}
		Choosing $\theta_1=\frac{1}{\sqrt{g}}$ proves the proposition. 
	\end{proof}
	
	\section{The Main Term}
	\begin{proposition}\thlabel{S(X)}
		For $X<g$, let 
		\[
		S(X)=\frac{1}{|\mathcal{P}_{2g+1}|}\sum_{P\in\mathcal{P}_{2g+1}}\sum_{f\in\mathbb{A}^+_{\leq 2X-1}}\frac{\tau(f)\chi_P(f)}{\sqrt{|f|}}(2g-\deg(f))^2.
		\]
		Then 
		\begin{align*}
			S(X)&=\frac{g^5}{15\zeta_{\mathbb{A}}(2)}+g^4\left(\frac{1}{2}+\frac{1}{6q}\right)+\frac{4g^3}{3}+g^2\left(\frac{3}{2}-\frac{1}{6q}\right)\\
			&-\frac{2g^2(g-X)^3}{3\zeta_{\mathbb{A}}(2)}-\frac{g^2(g-X)^2}{\zeta_{\mathbb{A}}(2)}-\frac{g^2(g-X)}{3\zeta_{\mathbb{A}}(2)}+O(q^{-g+X}g^2(g-X)^2)+O(g(g-X)^4).
		\end{align*}
	\end{proposition}
	\begin{proof}
		Split the sum into two sums, the first over $f=\square$ and the second over $f\neq \square$. Using \thref{Weilbound}, we have
		\begin{align*}
			\frac{1}{|\mathcal{P}_{2g+1}|}\sum_{P\in\mathcal{P}_{2g+1}}\sum_{\substack{f\in\mathbb{A}^+_{\leq 2g-1}\\f\neq \square}}\frac{\tau(f)\chi_P(f)}{\sqrt{|f|}}(2g-\deg(f))^2&\ll q^{-g+X}g^2(g-X)^2.
		\end{align*}
		Hence
		\begin{align*}
			S(X)&=\frac{1}{|\mathcal{P}_{2g+1}|}\sum_{P\in\mathcal{P}_{2g+1}}\sum_{\substack{f\in\mathbb{A}^+_{\leq 2X-1}\\f=\square}}\frac{\tau(f)\chi_P(f)}{\sqrt{|f|}}(2g-\deg(f))^2+O(q^{-g+X}g^2(g-X)^2)\\
			&=4\sum_{f\in\mathbb{A}^+_{\leq X-1}}\frac{\tau(f^2)}{|f|}(g-\deg(f))^2+O(q^{-g+X}g^2(g-X)^2).
		\end{align*}
		Using the function field analogue of Perron's formula and (\ref{sumtau}), we have 
		\begin{align}\label{S(X)sum}
			S(X)&=\frac{4}{2\pi i}\oint_{|u|=r}\sum_{f\in\mathbb{A}^+}\frac{\tau(f^2)}{|f|}u^{\deg(f)}\sum_{n=0}^{X-1}\frac{(g-n)^2}{u^n}\frac{du}{u}+O(q^{-g+X}g^2(g-X)^2)\nonumber\\
			&=\frac{4(g-X)^2}{2\pi i}\oint_{|u|=r}\frac{(1-\frac{u^2}{q})}{(1-u)^4u^X}du+\frac{8(g-X)}{2\pi i}\oint_{|u|=r}\frac{(1-\frac{u^2}{q})}{(1-u)^5u^X}du\nonumber\\&+\frac{4}{2\pi i}\oint_{|u|=r}\frac{(1-\frac{u^2}{q})(1+u)}{(1-u)^6u^X}du+O(q^{-g+X}g^2(g-X)^2)
		\end{align}
		Denote by $S_1(X)$, $S_2(X)$ and $S_3(X)$ the first, second, and third terms of (\ref{S(X)sum}) respectively. Enlarging the contour and evaluating the residue at $u=1$, we have 
		\[
		S_1(X)=\frac{2}{3}X^3(g-X)^2\left(1-\frac{1}{q}\right)+2X^2(g-X)^2\left(1+\frac{1}{q}\right)+O(X(g-X)^2),
		\]
		\[
		S_2(X)=\frac{1}{3}X^4(g-X)\left(1-\frac{1}{q}\right)+2X^3(g-X)\left(1+\frac{1}{3q}\right)+\frac{1}{3}X^2(g-X)\left(11+\frac{1}{q}\right)+O(X(g-X)),
		\]
		\begin{align*}
			S_3(X)=\frac{X^5}{15}\left(1-\frac{1}{q}\right)+X^4\left(\frac{1}{2}+\frac{1}{6q}\right)+\frac{4X^3}{3}+X^2\left(\frac{3}{2}-\frac{1}{6q}\right)+O(X).
		\end{align*}
		Using the notion $X=g-(g-X)$ and expanding, we have
		\begin{equation}\label{S1(X)}
			S_1(X)=\frac{2}{3}g^3(g-X)^2\left(1-\frac{1}{q}\right)-2g^2(g-X)^3+2g^2(g-X)^2\left(1+\frac{1}{q}\right)+O(g(g-X)^4),
		\end{equation}
		\begin{align}\label{S2(X)}
			S_2(X)&=\frac{1}{3}g^4(g-X)\left(1-\frac{1}{q}\right)-\frac{4}{3}g^3(g-X)^2\left(1-\frac{1}{q}\right)+2g^2(g-X)\left(1-\frac{1}{q}\right)\nonumber\\
			&+2g^3(g-X)\left(1+\frac{1}{3q}\right)-6g^2(g-X)^2\left(1+\frac{1}{3q}\right)+\frac{1}{3}g^2(g-X)\left(11+\frac{1}{q}\right)+O(g(g-X)^4),
		\end{align}
		\begin{align}\label{S3(X)}
			S_3(X)&=\frac{1}{15}g^5\left(1-\frac{1}{q}\right)-\frac{1}{3}g^4(g-X)\left(1-\frac{1}{q}\right)+\frac{2}{3}g^3(g-X)^2\left(1-\frac{1}{q}\right)-\frac{2}{3}g^2(g-X)^3\left(1-\frac{1}{q}\right)\nonumber\\
			&+g^4\left(\frac{1}{2}+\frac{1}{6q}\right)-4g^3(g-X)\left(\frac{1}{2}+\frac{1}{6q}\right)+6g^2(g-X)^2\left(\frac{1}{2}+\frac{1}{6q}\right)\nonumber\\&+\frac{4g^3}{3}-4g^2(g-X)+g^2\left(\frac{3}{2}-\frac{1}{6q}\right)+O(g(g-X)^4).
		\end{align}
		Combining (\ref{S1(X)}), (\ref{S2(X)}) and (\ref{S3(X)}) with (\ref{S(X)sum}) completes the proof of the proposition. 
	\end{proof}
	\section{Proof of \thref{Main}}
	\begin{proof}[Proof of \thref{Main}]
		Using \thref{AFELambda} we have
		\[
		\frac{1}{|\mathcal{P}_{2g+1}|}\sum_{P\in\mathcal{P}_{2g+1}}\frac{\Lambda^{''}\left(\frac{1}{2},\chi_P\right)\Lambda\left(\frac{1}{2},\chi_P\right)}{(\log q)^2}=\frac{1}{|\mathcal{P}_{2g+1}|}\sum_{P\in\mathcal{P}_{2g+1}}\sum_{f\in\mathbb{A}^+_{\leq 2g-1}}\frac{\tau(f)\chi_P(f)}{\sqrt{|f|}}(2g-\deg(f))^2
		\]
		Splitting the sum up at $\deg(f)=2X-1$, we have
		\[
		\frac{1}{|\mathcal{P}_{2g+1}|}\sum_{P\in\mathcal{P}_{2g+1}}\frac{\Lambda^{''}\left(\frac{1}{2},\chi_P\right)\Lambda\left(\frac{1}{2},\chi_P\right)}{(\log q)^2}=S(X)+E(X),
		\]
		where $S(X)$ and $E(X)$ are defined in \thref{S(X)} and \thref{E(X)} respectively. Then using \thref{S(X)} and \thref{E(X)} we have 
		\begin{align*}
			&\frac{1}{|\mathcal{P}_{2g+1}|}\sum_{P\in\mathcal{P}_{2g+1}}\frac{\Lambda^{''}\left(\frac{1}{2},\chi_P\right)\Lambda\left(\frac{1}{2},\chi_P\right)}{(\log q)^2}\\&=\frac{g^5}{15\zeta_{\mathbb{A}}(2)}+g^4\left(\frac{1}{2}+\frac{1}{6q}\right)+\frac{4g^3}{3}+g^2\left(\frac{3}{2}-\frac{1}{6q}\right)\\
			+&O_{\epsilon}(g^{\frac{3}{2}+\epsilon}(g-X)^3)+O(g^{\frac{1}{2}}(g-X)^5)+O(q^{-g+X}g^2(g-X)^2)+O(g(g-X)^4).
		\end{align*}
		Choosing $X=g-[100\log g]$ completes the proof.
	\end{proof}
	
	\section{Appendix - A weighted Sum involving $\sin$}
	
	In this section, we prove a weighted version of the double sum found in \cite[Lemma~9.4]{Florea2017fourth} that is essential for the proof of \thref{Main}. First we require two results, which are sums involving $\sin$ and $\cos$.  
	
	\begin{lemma}\cite[Lemma~9.2]{Florea2017fourth}\thlabel{sinecossum}
		Let $\theta$ be such that $\frac{1}{\theta}=o(g)$ and $k\geq 1$. Then 
		\[
		\sum_{m=1}^gm^k\sin(m\theta)=-\frac{g^k\cos((g+\frac{1}{2})\theta)}{2\sin(\frac{\theta}{2})}+O\left(g^{k-1}\frac{1}{\sin^2(\frac{\theta}{2})}\right)
		\]
		and
		\[
		\sum_{m=1}^gm^k\cos(m\theta)=\frac{g^k\sin((g+\frac{1}{2})\theta)}{2\sin(\frac{\theta}{2})}+O\left(g^{k-1}\frac{1}{\sin^2(\frac{\theta}{2})}\right).
		\]  
	\end{lemma}
	
	\begin{lemma}\cite[Lemma~9.3]{Florea2017fourth}\thlabel{sine/ksum}
		Let $\theta$ be such that $\frac{1}{\theta}=o(g)$. Then
		\[
		\sum_{k=1}^{a-1}\frac{\sin(k\theta)}{k}=\frac{\pi-\theta}{2}-\frac{\cos(a\theta)}{2a\sin(\frac{\theta}{2})}-\frac{\sin(a\theta)}{2a}+O\left(\frac{1}{a^2\sin^2(\frac{\theta}{2})}\right)
		\]  
	\end{lemma}

	\begin{lemma}\thlabel{weightedlemma}
		For $0\leq k\leq 9$ and $\frac{1}{\theta}=o(g)$, let 
		\[
		A(k,\theta):=\sum_{j=1}^{\alpha}j^2\sum_{m=0}^gm^k\frac{\sin(2\pi(m-j)\theta)}{m-j}+\sum_{j=1}^\alpha j^2\sum_{m=1}^gm^k\frac{\sin(2\pi(m+j)\theta)}{m+j}.
		\]
		Then
		\begin{align*}
			&A(k,\theta)\\&=\begin{cases}
				-g^{k-1}\left(\frac{\alpha^3}{3}+\frac{\alpha^2}{2}+\frac{\alpha}{6}\right)\left(\frac{\cos(2g\pi\theta)}{\pi \theta}-\sin(2g\pi\theta)\right)+O(\alpha^5g^{k-1}\theta+\alpha^4g^{k-2}\theta^{-1}) & \text{if }k\geq 2 \\
				-\left(\frac{\alpha^3}{3}+\frac{\alpha^2}{2}+\frac{\alpha}{6}\right)\left(\frac{\cos(2g\pi\theta)}{\pi \theta}-\sin(2g\pi\theta)-\cot(2\pi\theta)\right)\\
				+\sum_{m=1}^gm\sin(2\pi m\theta)(A(m)+B(m))+O(\theta \alpha^5)& \text{if } k=1\\
				\left(\frac{\alpha^3}{3}+\frac{\alpha^2}{2}+\frac{\alpha}{6}\right)\left(\pi-\frac{\cos(2g\pi\theta)}{g\pi\theta}+\frac{2\sin(2g\pi\theta)}{g} \right)+O(g^{-1}\theta\alpha^5+g^{-2}\theta^{-2}\alpha^3)& \text{if }k=0,
			\end{cases}
		\end{align*}
		where $A(m)$, $B(m)=O(\frac{\alpha^4}{m^2})$ which can be written down explicitly.
	\end{lemma}
	
	\begin{proof}
		Firstly, assume $k\geq 2$. By the addition formula for $\sin$ and the Taylor series for $\sin$ and $\cos$, we have
		\[
		\sin(2\pi (m\pm j)\theta)=\sin(2\pi m\theta)\pm 2\pi j\theta\cos(2\pi m\theta)+O(\alpha^2\theta^2)
		\]
		By \thref{sinecossum}, the error term is bounded by $g^{k-1}\alpha^5\theta$. Therefore
		\begin{align*}
			A(k,\theta)&=\sum_{m=1}^gm^k\sin(2\pi m\theta)\sum_{j=1}^{\alpha}\left(\frac{j^2}{m+j}-\frac{j^2}{m-j}\right)\\
			&+2\pi\theta\sum_{m=1}^gm^k\cos(2\pi m\theta)\sum_{j=1}^{\alpha}\left(\frac{j^3}{m+j}-\frac{j^3}{m-j}\right)+O(g^{k-1}\alpha^5\theta).
		\end{align*}
		Furthermore, using \thref{sinecossum} the second term is bounded by $g^{k-2}\alpha^5$. Hence
		\begin{equation}\label{A(k,theta)sum1}
			A(k,\theta)=\sum_{m=1}^gm^k\sin(2\pi m\theta)\sum_{j=1}^{\alpha}\left(\frac{j^2}{m+j}+\frac{j^2}{m-j}\right)+O(g^{k-1}\alpha^5\theta).
		\end{equation}
		Considering the sum
		\begin{align}\label{jsquaredsum1}
			\sum_{j=1}^\alpha\frac{j^2}{m+j}&=\sum_{\ell=m+1}^{m+\alpha}\frac{(\ell-m)^2}{\ell}=\sum_{\ell=m+1}^{m+\alpha}\ell-2m\sum_{\ell=m+1}^{m+\alpha}1+m^2\sum_{\ell=m+1}^{m+\alpha}\frac{1}{\ell}\nonumber\\
			&=\sum_{\ell=1}^{m+\alpha}\ell-\sum_{\ell=1}^m\ell-2m\sum_{\ell=1}^{m+\alpha}1+2m\sum_{\ell=1}^m1+m^2\sum_{\ell=1}^{m+\alpha}\frac{1}{\ell}-m^2\sum_{\ell=1}^m\frac{1}{\ell}\nonumber\\
			&=-m\alpha+\frac{\alpha^2+\alpha}{2}+m^2\left(\sum_{\ell=1}^{m+\alpha}\frac{1}{\ell}-\sum_{\ell=1}^m\frac{1}{\ell}\right)
		\end{align}
		Using the asymptotic expansion for harmonic numbers, we have 
		\[
		\sum_{\ell=1}^{m+\alpha}\frac{1}{\ell}=\log(m+\alpha)+\gamma+\frac{1}{2(m+\alpha)}-\sum_{k=1}^{\infty}\frac{B_{2k}}{2k}\frac{1}{(m+\alpha)^{2k}}
		\]
		and 
		\[
		\sum_{\ell=1}^m\frac{1}{\ell}=\log m+\gamma+\frac{1}{2m}-\sum_{k=1}^{\infty}\frac{B_{2k}}{2k}\frac{1}{m^{2k}},
		\]
		where $\gamma$ is the Euler–Mascheroni constant and $B_{2k}$ is the Bernoulli numbers. Therefore 
		\begin{align}\label{jsquaredsum2}
			\sum_{\ell=1}^{m+\alpha}\frac{1}{\ell}-\sum_{\ell=1}^{\alpha}\frac{1}{\ell}&=\log\left(1+\frac{\alpha}{m}\right)-\frac{\alpha}{2m(m+\alpha)}-\sum_{k=1}^{\infty}\frac{B_{2k}}{2k}\left(\frac{1}{(m+\alpha)^{2k}}-\frac{1}{m^{2k}}\right)\nonumber\\
			&=\sum_{k=1}^{\infty}\frac{(-1)^{k+1}\alpha^k}{km^k}-\frac{\alpha}{2m(m+\alpha)}-\sum_{k=1}^{\infty}\frac{B_{2k}}{2k}\left(\frac{1}{(m+\alpha)^{2k}}-\frac{1}{m^{2k}}\right)\nonumber\\
			&=\frac{\alpha}{m}-\frac{\alpha^2}{2m^2}+\frac{\alpha^3}{3m^3}-\frac{\alpha}{2m(m+\alpha)}+\frac{\alpha}{6m(m+\alpha)^2}+\frac{\alpha^2}{12m^2(m+\alpha)^2}\nonumber\\
			&+\sum_{k=4}^{\infty}\frac{(-1)^{k+1}\alpha^k}{km^k}-\sum_{k=2}^{\infty}\frac{B_{2k}}{2k}\left(\frac{1}{(m+\alpha)^{2k}}-\frac{1}{m^{2k}}\right).
		\end{align}
		Therefore combining (\ref{jsquaredsum1}) and (\ref{jsquaredsum2}), we have
		\begin{equation}\label{sumjm+j}
			\sum_{j=1}^{\alpha}\frac{j^2}{m+j}=\frac{\alpha^3}{3m}+\frac{\alpha^2}{2m}+\frac{\alpha}{6m}+A(m),
		\end{equation}
		where $A(m)=O\left(\frac{\alpha^4}{m^2}\right)$ that can be explicitly written down as 
		\begin{align*}
			A(m)&=\frac{\alpha^3}{2m(m+\alpha)}-\frac{\alpha^2}{3(m+\alpha)^2}+\frac{\alpha^3}{6m(m+\alpha)^2}-\frac{\alpha^2}{12(m+\alpha)^2}\\
			&+\sum_{k=4}^{\infty}\frac{(-1)^{k+1}\alpha^k}{km^{k-2}}-m^2\sum_{k=2}^{\infty}\frac{B_{2k}}{2k}\left(\frac{1}{(m+\alpha)^2}-\frac{1}{m^{2k}}\right).
		\end{align*}
		Similarly
		\begin{equation}\label{sumjm-j}
			\sum_{j=1}^{\alpha}\frac{j^2}{m-j}=\frac{\alpha^3}{3m}+\frac{\alpha^2}{2m}+\frac{\alpha}{6m}+B(m)
		\end{equation}
		where $B(m)=O\left(\frac{\alpha^4}{m^2}\right)$, where $B(m)$ can be written down explicitly and has a similar expression to $A(m)$. Therefore combining (\ref{sumjm+j}) and (\ref{sumjm-j}) with (\ref{A(k,theta)sum1}) we have 
		\begin{equation}
			A(k,\theta)=\left(\frac{2\alpha^3}{3}+\alpha^2+\frac{\alpha}{3}\right)\sum_{m=1}^gm^{k-1}\sin(2\pi m\theta)+\sum_{m=1}^gm^k(A(m)+B(m))\sin(2\pi m\theta)+O(g^{k-1}\theta\alpha^5).
		\end{equation}
		Using \thref{sinecossum}, the second term is bounded by $g^{k-2}\alpha^4\theta^{-1}$, hence using \thref{sinecossum} and the Taylor series for $\sin$ and $\cos$ we have
		\begin{align*}
			A(k,\theta)&=-g^{k-1}\left(\frac{\alpha^3}{3}+\frac{\alpha^2}{2}+\frac{\alpha}{6}\right)\left(\frac{\cos((g+\frac{1}{2})2\pi \theta)}{\sin(\pi \theta)}\right)+O(g^{k-1}\alpha^5\theta+g^{k-1}\alpha^4\theta^{-1})\\
			&=-g^{k-1}\left(\frac{\alpha^3}{3}+\frac{\alpha^2}{2}+\frac{\alpha}{6}\right)\left(\frac{\cos(2g\pi\theta)}{\pi\theta}-\sin(2g\pi\theta)\right)+O(g^{k-1}\alpha^5\theta+g^{k-1}\alpha^4\theta^{-1}).
		\end{align*}
		Assume now that $k=1$, then 
		\[
		A(1,\theta)=\sum_{j=1}^{\alpha}j^2\sum_{m=0}^gm\frac{\sin(2\pi(m-j)\theta)}{m-j}+\sum_{j=1}^{\alpha}j^2\sum_{m=1}^gm\frac{\sin(2\pi(m+j)\theta)}{m+j}.
		\]
		Using the same arguments seen in the case when $k\geq 2$ we have 
		\begin{equation}\label{A(1,theta)sum}
			A(1,\theta)=\sum_{m=1}^gm\sin(2\pi m\theta)\sum_{j=1}^{\alpha}\left(\frac{j^2}{m+j}+\frac{j^2}{m-j}\right)+O(\theta\alpha^5).
		\end{equation}
		Therefore combining (\ref{A(1,theta)sum}) with (\ref{sumjm+j}) and (\ref{sumjm-j}) we have 
		\begin{align*}
			A(1,\theta)&=\left(\frac{2\alpha^3}{3}+\alpha^2+\frac{\alpha}{3}\right)\sum_{m=1}^g\sin(2\pi m\theta)+\sum_{m=1}^gm\sin(2\pi m\theta)(A(m)+B(m))+O(\theta\alpha^5).
		\end{align*}
		From the proof of \thref{sinecossum}, (see in the proof of \cite[Lemma~9.2]{Florea2017fourth}) we have 
		\begin{align*}
			A(1,\theta)&=\left(\frac{\alpha^3}{3}+\frac{\alpha^2}{2}+\frac{\alpha}{6}\right)\left(\frac{\cos(\pi\theta)-\cos((g+\frac{1}{2})2\pi\theta)}{\sin(\pi\theta)}\right)\\
			&+\sum_{m=1}^gm\sin(2\pi m\theta)(A(m)+B(m))+O(\theta\alpha^5)\\
			&=-\left(\frac{\alpha^3}{3}+\frac{\alpha^2}{2}+\frac{\alpha}{6}\right)\left(\frac{\cos(2g\pi \theta)}{\pi\theta}-\sin(2g\pi\theta)-\cot(\pi\theta)\right)\\
			&+\sum_{m=1}^gm\sin(2\pi m\theta)(A(m)+B(m))+O(\theta\alpha^5).
		\end{align*}
		Finally, assume $k=0$, then 
		\[
		A(0,\theta)=\sum_{j=1}^{\alpha}j^2\sum_{m=1}^g\frac{\sin(2\pi(m+j)\theta)}{m+j}+\sum_{j=1}^{\alpha}j^2\sum_{m=0}^g\frac{\sin(2\pi(m-j)\theta)}{m-j}.
		\]
		Focusing on the first sum, we let $k=m+j$, then
		\begin{align}\label{A(0,theta)sum1}
			\sum_{j=1}^{\alpha}j^2\sum_{m=1}^g\frac{\sin(2\pi(m+j)\theta)}{m+j}&=\sum_{k=2}^{g+\alpha}\frac{\sin(2\pi k\theta )}{k}\sum_{j=\max\{1,k-g\}}^{\min\{\alpha,k-1\}}j^2\nonumber\\
			&=\sum_{k=2}^\alpha\frac{\sin(2\pi k\theta)}{k}\sum_{j=1}^{k-1}j^2+\sum_{k=\alpha+1}^{g+1}\frac{\sin(2\pi k\theta)}{k}\sum_{j=1}^{\alpha}j^2\nonumber\\
			&+\sum_{k=g+2}^{g+\alpha}\frac{\sin(2\pi k\theta)}{k}\sum_{j=k-g}^{\alpha}j^2
		\end{align}
		Similarly, considering the second sum and let $k=m-j$ we have 
		\begin{align}\label{A(0,theta)sum2}
			\sum_{j=1}^\alpha j^2\sum_{m=0}^g\frac{\sin(2\pi (m-j)\theta)}{m-j}&=\sum_{k=-\alpha}^{g-1}\frac{\sin(2\pi k\theta)}{k}\sum_{j=\max\{1,-k\}}^{\min\{\alpha,g-k\}}j^2\nonumber\\
			&=\sum_{k=-\alpha}^{-1}\frac{\sin(2\pi k\theta )}{k}\sum_{j=-k}^\alpha j^2+\sum_{k=0}^{g-\alpha}\frac{\sin(2\pi k\theta)}{k}\sum_{j=1}^{\alpha}j^2\nonumber\\
			&+\sum_{k=g-\alpha+1}^{g-1}\frac{\sin(2\pi k\theta)}{k}\sum_{j=1}^{g-k}j^2.
		\end{align}
		Combining (\ref{A(0,theta)sum1}) and (\ref{A(0,theta)sum2}) we have 
		\begin{align}\label{A(0,theta)eqncomb}
			A(0,\theta)&=\sum_{k=1}^{g+\alpha}\frac{\sin(2\pi k\theta)}{k}\sum_{j=1}^{\alpha}j^2+\sum_{k=0}^{g-\alpha}\frac{\sin(2\pi k\theta)}{k}\sum_{j=1}^{\alpha}j^2\nonumber\\
			&-\sum_{k=g+2}^{g+\alpha}\frac{\sin(2\pi k\theta)}{k}\sum_{j=1}^{k-g-1}j^2+\sum_{k=g-\alpha+1}^{g-1}\frac{\sin(2\pi k\theta)}{k}\sum_{j=1}^{g-k}j^2.
		\end{align}
		Invoking \thref{sine/ksum} we have 
		\begin{equation}\label{A(0,theta)main1}
			\sum_{k=1}^{g+\alpha}\frac{\sin(2\pi k\theta)}{k}=\frac{\pi -2\pi\theta}{2}-\frac{\cos((g+\alpha+1)2\pi \theta)}{2(g+\alpha+1)\sin(\pi\theta)}-\frac{\sin((g+\alpha+1)2\pi\theta)}{2(g+\alpha+1)}+O(g^{-2}\theta^{-2})
		\end{equation}
		and
		\begin{equation}\label{A(0,theta)main2}
			\sum_{k=0}^{g-\alpha}\frac{\sin(2\pi k\theta )}{k}=\frac{\pi +2\pi \theta}{2}-\frac{\cos((g-\alpha+1)2\pi\theta)}{2(g-\alpha+1)\sin(\pi\theta)}-\frac{\sin((g-\alpha+1)2\pi\theta)}{2(g-\alpha+1)}+O(g^{-2}\theta^{-2}).
		\end{equation}
		Using a change of variables in each of the final two terms of (\ref{A(0,theta)eqncomb}) we have that this is equal to 
		\begin{equation} \label{A(0,theta)bound}
			\sum_{y=1}^{\alpha-1}\left(\frac{y^3}{3}+\frac{y^2}{2}+\frac{y}{6}\right)\left(\frac{-(g-y)\sin(2\pi(g+1+y)\theta)+(g+1+y)\sin(2\pi(g-y)\theta)}{(g+1+y)(g-y)}\right)=O(g^{-1}\theta\alpha^5).    
		\end{equation}
		
		Hence, combining (\ref{A(0,theta)main1}), (\ref{A(0,theta)main2}) and (\ref{A(0,theta)bound}) with (\ref{A(0,theta)eqncomb}) we obtain
		\begin{align*}
			A(0,\theta)&=\left(\frac{\alpha^3}{3}+\frac{\alpha^2}{2}+\frac{\alpha}{6}\right)\left(\pi -\frac{\cos(2\pi g\theta)}{g\sin(\pi \theta)}+\frac{2\sin(2\pi g \theta)}{g}\right)+O(g^{-2}\theta^{-2}\alpha^3)+O(g^{-1}\theta\alpha^5)\\
			&=\left(\frac{\alpha^3}{3}+\frac{\alpha^2}{2}+\frac{\alpha}{6}\right)\left(\pi -\frac{\cos(2\pi g\theta)}{g\pi\theta}+\frac{2\sin(2\pi g \theta)}{g}\right)+O(g^{-2}\theta^{-2}\alpha^3)+O(g^{-1}\theta\alpha^5)
		\end{align*}
	\end{proof}


\begin{thebibliography}{}
		
		\bibitem{AndradeJung2021} J.~C. Andrade and H.~Y. Jung, {\emph Mean values of derivatives of $L$-functions in function fields: IV}, J. Korean Math. Soc. {\bf 58}(6) (2021), no.~6, 1529-1547.
		
		\bibitem{AndradeJungShamesaldeen2021} J.~C. Andrade, H.~Y. Jung and A. Shamesaldeen, {\emph The integral moments and ratios of quadratic Dirichlet $L$-functions over monic irreducible polynomials in $\mathbb{F}_q[T]$}, Ramanujan J. {\bf 56}(1) (2021), 23-66
		
		\bibitem{AndradeKeating2012} J.~C. Andrade and J.~P. Keating, \emph{The mean value of $L(\frac12,\chi)$ in the hyperelliptic ensemble}, J. Number Theory {\bf 132}(12) (2012), 2793-2816.
		
		\bibitem{AndradeKeating2013}J.C. Andrade and J.P. Keating, {\emph Mean value theorems for $L$-functions over prime polynomials for the rational function field}, Acta Arith. {\bf 161}(4) (2013), 371-385.
		
		\bibitem{AndradeKeating2014} J.~C. Andrade and J.~P. Keating, {\emph Conjectures for the integral moments and ratios of $L$-functions over function fields}, J. Number Theory {\bf 142} (2014), 102-148.
		
		\bibitem{AndradeRajagopal2016} J.~C. Andrade and S. Rajagopal, \emph{Mean values of derivatives of $L$-functions in function fields: I}, J. Math. Anal. Appl. {\bf 443}(1) (2016), 526-541.
		
		\bibitem{BaeJung2019} S. Bae and H.~Y. Jung, \emph{Note on the mean values of derivatives of quadratic Dirichlet $L$-functions in function fields}, Finite Fields Appl. {\bf 57} (2019), 249-267.
		
		\bibitem{BaluyotPratt2022} S. Baluyot and K. Pratt, Dirichlet $L$-functions of quadratic characters of prime conductor at the central point, J. Eur. Math. Soc. (JEMS) {\bf 24}(2) (2022), 369-460.
		
		\bibitem{Best2026} C.~G. Best, {\emph Moments of derivatives of quadratic Dirichlet $L$-functions with prime conductor}, Res. Number Theory {\bf 12}(1) (2026), 1-17.
		
		\bibitem{BuiFlorea2020} H. Bui and A. Florea, \emph{Moments of Dirichlet $L$-functions with prime conductors over function fields}, Finite Fields Appl. {\bf 64} (2020), 101659.
		
		\bibitem{CFKRS2005} J.~B. Conrey, D.~W. Farmer, and J.~P. Keating,
		M. ~O. Rubinstein,  and N.~C. Snaith, \emph{Integral moments of $L$-functions}, Proc. London Math. Soc. (3) {\bf 91}(1) (2005), 33-104.
		
		\bibitem{DjankovicDjokic2021} G. Djankovi\'c{} and D. {\DJ}oki\'c, \emph{The mixed second moment of quadratic Dirichlet $L$-functions over function fields}, Rocky Mountain J. Math. {\bf 51}(6) (2021), 2003-2017.
		
		\bibitem{Florea2017first} A.~M. Florea, {\emph Improving the error term in the mean value of $L(\frac12,\chi)$ in the hyperelliptic ensemble}, Int. Math. Res. Not. IMRN {\bf 2017}(20), 6119-6148.
		
		\bibitem{Florea2017secondthird} A.~M. Florea, {\emph The second and third moment of $L(1/2,\chi)$ in the hyperelliptic ensemble}, Forum Math. {\bf 29}(4) (2017), 873-892.
		
		\bibitem{Florea2017fourth} A.~M. Florea, \emph{The fourth moment of quadratic Dirichlet L-functions over function fields}, Geom. Funct. Anal., {\bf 27}, (2017), 541-595.
		
		\bibitem{HardyLittlewood1916} G.~H. Hardy and J.~E. Littlewood, {\emph Contributions to the theory of the riemann zeta-function and the theory of the distribution of primes}, Acta Math. {\bf 41}(1) (1916), 119-196.
		
		\bibitem{Ingham1927} A.~E. Ingham, {\emph Mean-Value Theorems in the Theory of the Riemann Zeta-Function}, Proc. London Math. Soc. (2) {\bf 27}(4) (1927), 273-300. 
		
		\bibitem{Jung2022} H.~Y. Jung, {\emph Mean values of derivatives of quadratic prime Dirichlet $L$-functions in function fields}, Commun. Korean Math. Soc. {\bf 37}(3) (2022), 635-648.
		
		\bibitem{Jutila1981} M. Jutila, \emph{On the mean value of $L(\frac{1}{2},\,\chi )$\ for real characters}, Analysis {\bf 1}(2) (1981), 149-161.
		
		\bibitem{KeatingSnaith2000R} J.~P. Keating and N.~C. Snaith, \emph{Random matrix theory and $\zeta(1/2+it)$}, Comm. Math. Phys. {\bf 214}(1) (2000), 57-89.
		
		\bibitem{KeatingSnaith2000L}J.~P. Keating and N.~C. Snaith, \emph{Random matrix theory and $L$-functions at $s=1/2$}, Comm. Math. Phys. {\bf 214}(1) (2000), 91-110.
		
		\bibitem{KeatingWei2024} J.~P. Keating and F. Wei, \emph{Joint moments of higher order derivatives of CUE characteristic polynomials I: asymptotic formulae}, Int. Math. Res. Not. IMRN {\bf 2024}(12), 9607-9632.
		
		\bibitem{Rosen2002} M. Rosen, \emph{Number Theory in Function Fields}, Graduate Texts in Mathematics, Vol. 210, Springer-Verlang, New York, (2002).
		
		\bibitem{Rudnick2010} Z. Rudnick, {\emph Traces of high powers of the Frobenius class in the hyperelliptic ensemble}, Acta Arith. {\bf 143}(1) (2010), 81-99.
		
		\bibitem{Shen2021} Q. Shen, \emph{The fourth moment of quadratic Dirichlet $L$-functions}, Math. Z. {\bf 298}(2) (2021), 713-745.
		
		\bibitem{ShenStucky2026} Q. Shen and J. Stucky, \emph{The fourth moment of quadratic Dirichlet L-Functions II}, accepted to appear in Algebra and Number Theory (2026).
		
		\bibitem{Soundararajan2000} K. Soundararajan, \emph{Nonvanishing of quadratic Dirichlet $L$-functions at $s=\frac12$}, Ann. of Math. (2) {\bf 152}(2) (2000), 447-488
		
		\bibitem{Weil1948} A. Weil, {\emph Sur les courbes alg\'ebriques et les vari\'et\'es qui s'en d\'eduisent}, vol. 7. Hermann \& Cie, Paris (1948).
		
	\end{thebibliography}
\end{document}